\documentclass[letterpaper, 10 pt, conference]{ieeeconf}
\IEEEoverridecommandlockouts
\usepackage{cite}
\usepackage{amsmath,amssymb,amsfonts}
\usepackage{algorithmic}
\usepackage{graphicx}
\usepackage{textcomp}
\usepackage{amsthm}
\usepackage{subfigure}
\usepackage{color}

\def\BibTeX{{\rm B\kern-.05em{\sc i\kern-.025em b}\kern-.08em
    T\kern-.1667em\lower.7ex\hbox{E}\kern-.125emX}}
\newtheorem{theorem}{Theorem}[]
\newtheorem{lemma}[theorem]{Lemma}
\newtheorem{proposition}[theorem]{Proposition}

\theoremstyle{definition}

\begin{document}

\title{Optimal Geodesic Curvature Constrained Dubins' Path on Sphere with Free Terminal Orientation\\ 
\thanks{$^{1}$Deepak Prakash Kumar and Swaroop Darbha are with the Department of Mechanical Engineering, Texas A\&M University, College Station, TX, 77843, USA (e-mail: {\tt\footnotesize deepakprakash1997@gmail.com, dswaroop@tamu.edu})        }%
\thanks{$^{2}$Satyanarayana G Manyam is with Infoscitex corporation, Dayton, OH 45431, USA
(e-mail: {\tt\footnotesize msngupta@gmail.com})}%
\thanks{$^{3}$Dzung Tran is with the University of Dayton Research Institute, Dayton, OH 45469 USA (e-mail:
{\tt\footnotesize dzung.tran.ctr@afrl.af.mil})}
\thanks{$^{4}$David Casbeer is with the Autonomous Control Branch, Air Force Research Laboratory, Wright-Patterson Air Force Base, OH 45433 USA (e-mail:
{\tt\footnotesize david.casbeer@us.af.mil})}
\thanks{Distribution Statement A: Approved for Public Release; Distribution is Unlimited. PA\# AFRL-2023-3116}
}

\author{Deepak Prakash Kumar$^{1}$ \and Swaroop Darbha$^{1}$ \and Satyanarayana Gupta Manyam$^{2}$ \and Dzung Tran$^{3}$ \and David W. Casbeer$^{4}$}

\maketitle

\begin{abstract}
In this paper, motion planning for a vehicle moving on a unit sphere with unit speed is considered, wherein the desired terminal location is fixed, but the terminal orientation is free. The motion of the vehicle is modeled to be constrained by a maximum geodesic curvature $U_{max},$ which controls the rate of change of heading of the vehicle such that the maximum heading change occurs when the vehicle travels on a tight circular arc of radius $r = \frac{1}{\sqrt{1 + U_{max}^2}}$. Using Pontryagin's Minimum Principle, the main result of this paper shows that for $r \leq \frac{1}{2}$, the optimal path connecting a given initial configuration and a final location on the sphere belongs to a set of at most seven paths. The candidate paths are of type $CG, CC,$ and degenerate paths of the same, where $C \in \{L, R\}$ denotes a tight left or right turn, respectively, and $G$ denotes a great circular arc.
\end{abstract}

\section{Introduction}

Motion (or path) planning for autonomous vehicles, particularly unmanned aerial vehicles (UAVs), involves connecting given initial and final configurations with minimum cost and is of significant interest for military and civilian applications. The path planning problem in a plane is typically addressed by modeling the vehicle to be a Dubins vehicle, wherein the vehicle travels forward at a unit speed and has a bound on the curvature \cite{Dubins}. The classical Markov-Dubins problem was solved in \cite{Dubins}, wherein the authors showed that the candidate curvature-constrained paths that connect given initial and final configurations on a plane with minimum length is of type $CSC,$ $CCC$, and corresponding degenerate paths. Here, $C \in \{L, R \}$ denotes a tight left or right turn, respectively, of radius corresponding to the curvature bound, and $S$ denotes a straight line segment. A variation in the problem, wherein the vehicle can move forwards or backward with atmost unit speed, called a Reeds-Shepp vehicle, was studied in \cite{Reeds_Shepp}.

The results in both \cite{Dubins} and \cite{Reeds_Shepp} were obtained without using Pontryagin's Minimum Principle (PMP) \cite{PMP}. Studies such as \cite{boissonat} and \cite{sussman_geometric_examples} utilize PMP to characterize the optimal path systematically for these problems. Recently, in \cite{phase_portrait_kaya}, PMP was combined with a phase portrait approach to further simplify the proofs for the Markov-Dubins problem. 
A generalization of the classical Markov-Dubins problem, called the weighted Markov-Dubins problem, was studied in \cite{weighted_Markov_Dubins}, wherein the sinistral and dextral curvature bounds may be different and turns incur a penalty. The authors employed PMP and phase portraits to obtain candidate paths of at most five segments.

Few studies have considered path planning in 3D. In \cite{monroy}, the author considered path planning for a Dubins vehicle on a Riemannian manifold. In particular, the author showed that the Dubins result on 2D extends to a unit sphere for $r = \frac{1}{\sqrt{2}},$ where $r$ is the radius of the tight turn. In \cite{sussman_3D}, a {\it total} curvature-constrained 3D Dubins problem was considered; only the tangential direction of motion and the location were considered specified at the initial and final configurations. The author showed that the optimal path is either a helicoidal arc or a concatenation of at most three segments. Studies such as \cite{time_optimal_synthesis_SO3} and \cite{time_optimal_control_satellite} consider a generic time-optimal problem on $SO (3)$ with one and two control inputs, respectively, but assume the bound on the control input to be fixed to be equal to one. In \cite{3D_Dubins_sphere}, the authors modeled the Dubins problem on a sphere as a geodesic-curvature constrained shortest-path problem. To this end, the vehicle was modeled using a Sabban frame. The authors showed that for $r \leq \frac{1}{2}$ and for normal controls (wherein the adjoint variable to the integrand in the objective functional is non-zero), the optimal path is of type $CGC,$ $CCC,$ or a degenerate path of the same. Here, $C \in \{L, R\}$ and $G$ denote a tight left or right turn and a great circular arc
, respectively.

From the papers surveyed, it can be observed that no study considered a motion planning problem on a sphere with a prescribed final location but with free final orientation. 
The main contribution of this paper shows that the optimal path for this problem is $LG, RG, LR, RL,$ or a degenerate path of the same for $r \leq \frac{1}{2}.$ Further, for $LR$ and $RL$ paths, the angle of the final segment is greater than or equal 
to $\pi$. The proofs for the same build on proofs and results from \cite{3D_Dubins_sphere}.

\section{Problem Formulation}

In the considered problem, the desired final configuration of the vehicle on the sphere may be partially specified, i.e., only the location is specified. In that case, one would expect a simplification of the characterization of the optimal path compared to the path planning problem on a sphere with a given final configuration \cite{3D_Dubins_sphere}. The formal specification of the problem is as follows:
\begin{align} \label{eq: objective_functional}
    J = \min \quad \int_0^L 1 \; ds,
\end{align}
subject to 
\begin{align}
    \frac{d{\mathbf X}}{ds} &= {\mathbf T}(s), \label{eq: X_evolution} \\
    \frac{d{\mathbf T}}{ds} &= -{\bf X}(s) + u_g(s) {\mathbf N}(s), \label{eq: T_evolution} \\
    \frac{d{\mathbf N}}{ds} &= -u_g(s) {\mathbf T}(s), \label{eq: N_evolution}
\end{align}
with boundary conditions
\begin{align} \label{eq: boundary_conditions}
    \mathbf{R}(0) = I, \quad \mathbf{R} (L) e_1 = \mathbf{X}_f,
\end{align}
where $e_1$ is the first column of the identity matrix $I$, and $\mathbf{X}_f$ is a unit vector indicating the desired final location on the unit sphere, which is the first column of the desired final orientation. In \eqref{eq: X_evolution}, \eqref{eq: T_evolution}, and \eqref{eq: N_evolution}, $\mathbf{X}, \mathbf{T},$ and $\mathbf{N}$ denote the position vector, tangent vector, and the tangent-normal vectors, respectively. The three vectors form an orthonormal basis on the sphere, as shown in Fig.~\ref{fig: sphere_model} (from \cite{3D_Dubins_sphere}), and describe the location and orientation of the vehicle. Hence, the rotation matrix $\mathbf{R} = [\mathbf{X}, \mathbf{T}, \mathbf{N}]$ is constructed using these three vectors, using which the boundary conditions are specified in \eqref{eq: boundary_conditions}. Further, the control input $u_g \in [-U_{max}, U_{max}]$ is the geodesic curvature \cite{3D_Dubins_sphere}.
\begin{figure}[htb!]
    \centering
    \includegraphics[width = 0.5\linewidth]{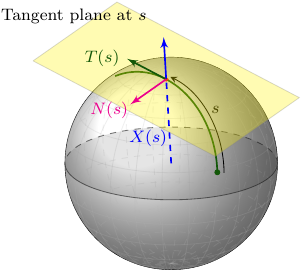}
    \caption{Dubins model for motion on a sphere \cite{3D_Dubins_sphere}}
    \label{fig: sphere_model}
\end{figure}

The Hamiltonian is defined through the adjoint variables $e, \lambda_1 (s), \lambda_2 (s), \lambda_3 (s)$ to apply Pontryagin's minimum principle (PMP) as
\begin{align*}
H &= e + \langle \lambda_1, {\mathbf T} \rangle + \langle\lambda_2, -{\mathbf X} + u_g {\mathbf N} \rangle + \langle\lambda_3, - u_g {\mathbf T} \rangle \\
&= e + \langle\lambda_1, {\mathbf T}\rangle - \langle \lambda_2, {\mathbf X} \rangle + u_g \left(\langle\lambda_2, {\mathbf N}\rangle - \langle\lambda_3, {\mathbf T}\rangle \right),
\end{align*}
where $e \geq 0$ from PMP. In the above equation, the explicit dependence of $H$ on $e, s,$ and the adjoint variables is not shown for brevity. Define:
\begin{align}
    A &:= \langle\lambda_2, {\mathbf N}\rangle - \langle\lambda_3, {\mathbf T}\rangle, \label{eq: definition_A} \\
    B &:= \langle\lambda_3, {\mathbf X}\rangle - \langle\lambda_1, {\mathbf N}\rangle, \\
    C &:= \langle\lambda_1, {\mathbf T}\rangle - \langle\lambda_2, {\mathbf X}\rangle.
\end{align}
Accordingly, 
\begin{align}
\label{eq: adjoint}
\frac{dA}{ds} = B, \quad \quad \frac{dB}{ds} = - A + u_g C, \quad \quad \frac{dC}{ds} = - u_gB, 
\end{align}
and 
\begin{align} \label{eq: Hamiltonian}
    H = e + C + u_g A.
\end{align}
Since the evolution equations in \eqref{eq: X_evolution}, \eqref{eq: T_evolution}, \eqref{eq: N_evolution} are autonomous, and the considered problem is free-time, the optimal control actions correspond to $H\equiv 0$ \cite{PMP_lecture_notes}.

\begin{lemma} \label{lemma: control_actions}
    The optimal control actions are given by $u_g \in \{-U_{max}, 0, U_{max} \}$ for $e > 0$. For $e = 0,$ $u_g \in \{-U_{max}, U_{max}\}$ if $A \not\equiv 0.$
\end{lemma}
\begin{proof}
    Since $H$ is linear in $u_g$ and $u_g$ pointwise minimizes $H$ from PMP, $u_g = -U_{max} \text{sgn} (A)$ when $A \neq 0$. Here, sgn is the sign function. Suppose $A \equiv 0.$ Then, $B = \frac{dA}{ds} \equiv 0$, and $\frac{dC}{ds} = -u_g B \equiv 0.$ Hence, $C$ is a constant. Since $H \equiv 0,$ $C = -e$ from \eqref{eq: Hamiltonian}. 
    If $e > 0,$ then $u_g \equiv 0$ as $\frac{dB}{ds} = -A + u_g C = 0 - u_g e \equiv 0$ (since $B \equiv 0$). 
\end{proof}

It should be noted here that the segment corresponding to $u_g = U_{max}$ is a left tight turn of radius $r = \frac{1}{\sqrt{1 + U_{max}^2}},$ whereas for $u_g = -U_{max},$ the segment is a right tight turn of the same radius. Further, corresponding to $u_g = 0,$ the segment is an arc of a great circle \cite{3D_Dubins_sphere}. Henceforth, $C$ will denote a tight turn, $L$ and $R$ will denote the left and right tight turns, respectively, and $G$ will denote a great circular arc.

Since the terminal orientation, i.e., $\mathbf{T} (L)$ and $\mathbf{N} (L),$ are free, the corresponding adjoint variables, which are $\lambda_2$ and $\lambda_3$, are zero at $s = L$. Since $\lambda_2 (L) = \lambda_3 (L) = 0,$ $A (L) = 0$ from its definition in \eqref{eq: definition_A}. This particular terminal condition of $A$ will be repeatedly used in this paper to obtain the candidate paths for the considered problem.

\section{Characterizing the Optimal Path}

In this section, the following main result of the paper will be proved:
\begin{theorem} \label{theorem: optimal_paths}
    For $r \leq \frac{1}{2},$ the optimal paths for free terminal orientation are $LG, RG, LR, RL,$ and degenerate paths of the same. Further, for paths $LR$ and $RL$, the angle of the final segment is greater than or equal to $\pi$.
\end{theorem}
To this end, the lemmas that follow in this section will be utilized.

For the considered problem, the optimal path will be characterized using PMP. Since $e \geq 0,$ two cases can be considered depending on the value of $e:$ $e = 0,$ and $e > 0$. In the former case, the obtained solutions are abnormal solutions since they are independent of the objective functional in \eqref{eq: objective_functional}. This particular case will be addressed using the phase portrait of $A$. In the latter case, $e > 0$ can be replaced with $e = 1$ without loss of generality since all other adjoint variables will be scaled accordingly. Since the candidate paths for prescribed initial and final configurations on a sphere were obtained in \cite{3D_Dubins_sphere} for $e = 1$, the obtained list of candidate paths will be used for the free terminal orientation problem and further reduced. It should be noted that if a path is not optimal for given initial and final configurations on a sphere, it cannot be optimal for the free terminal orientation problem.

\subsection{Candidate paths for $e = 0$}

Since the control action depends on $A,$ the optimal path can be characterized by analyzing the evolution of $A$. Since $A$ is differentiable, and $\frac{dA}{ds} = B,$ $\frac{dA}{ds}$ is differentiable. Hence, the phase portrait of $A$ can be used to determine the candidate paths. Using \eqref{eq: adjoint}, 
\begin{align} \label{eq: evolution_d2Ads2}
    \frac{dB (s)}{ds} = - A (s) + u_g (s) C (s) = \frac{d^2 A (s)}{d s^2}.
\end{align}
Since $H \equiv 0$ for optimal control, $C (s)$ can be obtained from \eqref{eq: Hamiltonian} as
\begin{align}
    C (s) = -e - u_g (s) A (s) = -u_g (s) A (s).
\end{align}
Substituting the above equation in \eqref{eq: evolution_d2Ads2}, the evolution of $A (s)$ is obtained as
\begin{align} \label{eq: dynamics_A_e_zero}
    \frac{d^2 A (s)}{d s^2} + \left(1 + u_g^2 (s) \right) A (s) &= 0.
\end{align}
If $A < 0$ or $A > 0,$ $u_g^2 (s) = U_{max}^2$ from Lemma~\ref{lemma: control_actions}. Hence, the solution for $A (s)$ and $\frac{dA}{ds} (s)$ is given by
\begin{align}
    A (s) &= \lambda \sin{\left(\sqrt{1 + U_{max}^2} s - \phi \right)}, \label{eq: evolution_A_e_0} \\
    \frac{dA}{ds} (s) &= \lambda \sqrt{1 + U_{max}^2} \cos{\left(\sqrt{1 + U_{max}^2} s - \phi \right)}, \label{eq: evolution_dA_ds_e_0}
\end{align}
where $\lambda > 0$. As $A (s)$ and $\frac{dA}{ds} (s) = B (s)$ are continuous and $\lambda > 0$ is considered, $A (s) \rightarrow 0$ implies that $\sin{\left(\sqrt{1 + U_{max}^2} s - \phi \right)} \rightarrow 0.$ Therefore, $\lim_{A (s) \rightarrow 0} \frac{dA}{ds} (s) = \pm \lambda \sqrt{1 + U_{max}^2} \neq 0,$ since $\lambda > 0$. Hence, $A (s) = 0$ transiently. As $A (s) < 0 \implies u_g (s) \equiv U_{max}$ and $A (s) > 0 \implies u_g (s) \equiv -U_{max},$ $\left(A (s), \frac{dA (s)}{ds} \right) = (0, \pm \lambda \sqrt{1 + U_{max}^2})$ correspond to inflection points between $L$ and $R$ segments.

The phase portrait for $A (s)$ can be obtained using a function that relates $A (s)$ and $\frac{dA (s)}{ds}$, which is given by
\begin{align} \label{eq: phase_portrait_e_0}
    f \left(A (s), \frac{dA (s)}{ds} \right) := A^2 (s) + \frac{1}{1 + U_{max}^2} \left(\frac{dA (s)}{ds} \right)^2 = \lambda^2.
\end{align}
Using the above-defined function, which is satisfied by the obtained solutions for $A (s)$ and $\frac{dA}{ds} (s),$ the phase portrait obtained for $e = 0$ is shown in Fig.~\ref{fig: phase_portrait_A_e_0}.

\begin{figure}[htb!]
    \centering
    \includegraphics[width = 0.6\linewidth]{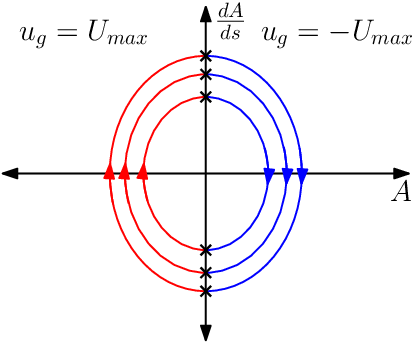}
    \caption{Phase portrait of $A$ for $e = 0$}
    \label{fig: phase_portrait_A_e_0}
\end{figure}

\begin{proposition} \label{prop: optimal_path_e_0}
For $e = 0,$ the optimal path is a concatenation of $C$ segments.
\end{proposition}

\begin{lemma} \label{lemma: angle_C_segment_e_0}
    For $e = 0,$ the angle of a $C$ segment that is completely traversed is exactly $\pi$ radians. 
\end{lemma}
\begin{proof}
    Consider an $L$ segment that is completely traversed. Let $s$ denote the arc length of the segment. Hence, from Fig.~\ref{fig: phase_portrait_A_e_0} and \eqref{eq: phase_portrait_e_0}, $A (0) = 0, \frac{dA}{ds} (0) = -\lambda\sqrt{1 + U_{max}^2},$ and $A (s) = 0, \frac{dA}{ds} (s) = \lambda \sqrt{1 + U_{max}^2}.$ Hence, from the evolution of $A$ and $\frac{dA}{ds}$ given in \eqref{eq: evolution_A_e_0} and \eqref{eq: evolution_dA_ds_e_0}, respectively,
    \begin{align*}
        \sin{(-\phi)} &= \sin{\left(\sqrt{1 + U_{max}^2} s - \phi \right)} = 0, \\
        \cos{(-\phi)} &= - \cos{\left(\sqrt{1 + U_{max}^2} s - \phi \right)} = - 1. 
    \end{align*}
    Hence, $\sqrt{1 + U_{max}^2} s = \pi$. However, since the radius of the $L$ segment is $r = \frac{1}{\sqrt{1 + U_{max}^2}},$ the angle of the $C$ segment is equal to $\pi$. A similar proof follows for an $R$ segment.
\end{proof}

Since the problem considers free terminal orientation, $A (L) = 0$. Therefore, using the phase portrait in Fig.~\ref{fig: phase_portrait_A_e_0}, Proposition~\ref{prop: optimal_path_e_0}, and Lemma~\ref{lemma: angle_C_segment_e_0}, it follows that the optimal path is of type $C_\alpha C_\pi C_\pi \cdots C_\pi$, where $\alpha \in (0, \pi]$, since the last $C$ segment must also be completely traversed.

\begin{lemma}
    For $e = 0,$ the optimal path is of type $C_\alpha C_\pi$ for $r \leq \frac{1}{\sqrt{2}}$, where $\alpha < \pi$, or a degenerate path of the same.
\end{lemma}
\begin{proof}
    Consider a $L_\pi R_\pi$ path. It is claimed that there exists an alternate $LG$ path of lower length connecting the same initial configuration and final location as the considered $L_\pi R_\pi$ path for $r \leq \frac{1}{\sqrt{2}}$. To this end, the alternate path must satisfy
    \begin{align*}
        \mathbf{R}_L (r, \pi) \mathbf{R}_R (r, \pi) e_1 = \mathbf{R}_L (r, \phi_1) \mathbf{R}_G (\phi_2) e_1,
    \end{align*}
    where $\mathbf{R}_L, \mathbf{R}_R, \mathbf{R}_G$ denote the rotation matrices corresponding to the $L, R,$ and $G$ segments, respectively. Using the expressions derived for these rotation matrices in Appendix~\ref{app: derivation_rot_mat}, the above equation can be simplified as
    \begin{align} \label{eq: equation_CpiCpi}
    \begin{split}
        &\begin{pmatrix}
            1 + 8 r^4 - 8 r^2 \\
            0 \\
            4 \left(1 - 2 r^2 \right) r \sqrt{1 - r^2}
        \end{pmatrix} \\
        &= \begin{pmatrix}
            c \phi_2 \left(1 - (1 - c \phi_1) r^2 \right) - r s \phi_1 s \phi_2 \\
            r s \phi_1 c \phi_2 + c \phi_1 s \phi_2 \\
            (1 - c \phi_1) r \sqrt{1 - r^2} c \phi_2 + s \phi_1 s \phi_2 \sqrt{1 - r^2}
        \end{pmatrix}.
    \end{split}
    \end{align}
    From the third equation,
    \begin{align}
        r (1 - c \phi_1) c \phi_2 + s \phi_1 s \phi_2 = 4 (1 - 2 r^2) r. \label{eq: third_equation_CpiCpi}
    \end{align}
    Substituting the above equation in the first equation in \eqref{eq: equation_CpiCpi} and simplifying, $c \phi_2 = 1 - 4 r^2$.
    Hence, for $r \leq \frac{1}{\sqrt{2}},$ a solution for $\phi_2$ can be obtained. From the two possible solutions, the solution $\phi_2 = \cos^{-1} \left(1 - 4 r^2 \right) \in [0, \pi]$ for $r \leq \frac{1}{\sqrt{2}}$ is chosen.

    Dividing the second equation in \eqref{eq: equation_CpiCpi} by $\sqrt{r^2 c^2 \phi_2 + s^2 \phi_2}$ for $0 < r \leq \frac{1}{\sqrt{2}}$, and defining $s \gamma := \frac{r c \phi_2}{\sqrt{r^2 c^2 \phi_2 + s^2 \phi_2}}, c \gamma := \frac{s \phi_2}{\sqrt{r^2 c^2 \phi_2 + s^2 \phi_2}}$, the equation can be rewritten as
    \begin{align*}
        s \gamma s \phi_1 + c \gamma c \phi_1 = \cos{\left(\phi_1 - \gamma \right)} &= 0.
    \end{align*}
    Among the two possible solutions, $\phi_1 = \frac{\pi}{2} + \gamma$ is chosen. It is claimed that $\gamma \in \left[\frac{-\pi}{2}, \tan^{-1} \left(\frac{1}{2 \sqrt{2}} \right) \right)$ for $0 < r \leq \frac{1}{\sqrt{2}}$, which in turn implies that $\phi_1 \geq 0$. 
    For this purpose, the expression for $\sqrt{r^2 c^2 \phi_2 + s^2 \phi_2}$ is first obtained by substituting the expression for $\phi_2$ as $\sqrt{r^2 c^2 \phi_2 + s^2 \phi_2} = \sqrt{r^2 (4 r^2 - 3)^2} = r (3 - 4 r^2)$,
    since $0 < r \leq \frac{1}{\sqrt{2}}$ is considered. Hence,
    \begin{align}
        s \gamma &= \frac{r c \phi_2}{\sqrt{r^2 c^2 \phi_2 + s^2 \phi_2}} = \frac{1 - 4 r^2}{3 - 4 r^2}, \label{eq: sin_gamma} \\
        c \gamma &= \frac{s \phi_2}{\sqrt{r^2 c^2 \phi_2 + s^2 \phi_2}} = \frac{2 \sqrt{2} \sqrt{1 - 2 r^2}}{3 - 4 r^2}. \label{eq: cos_gamma}
    \end{align}
    Since the expression for $s \gamma$ in terms of $r$ is smooth for $0 < r \leq \frac{1}{\sqrt{2}},$ $\frac{d s \gamma}{d r}$ is obtained as $\frac{-16r}{(3 - 4 r^2)^2} \leq 0$. Hence, the supremum of $s\gamma$ is at $r \rightarrow 0^+,$ and the minimum is at $r = \frac{1}{\sqrt{2}}.$ Hence, the claim follows by evaluating $s \gamma$ at the respective values for $r$.

    Since $\phi_2$ was obtained using the first equation in \eqref{eq: equation_CpiCpi} and $\phi_1$ was obtained using the second equation, only \eqref{eq: third_equation_CpiCpi} remains to be verified for the obtained solution. Substituting the obtained expressions in the LHS of \eqref{eq: third_equation_CpiCpi}, noting that $c \phi_1 = -s \gamma, s \phi_1 = c \gamma$, and using the expressions for $s \gamma$ and $c \gamma$ in terms of $\phi_2$,
    \begin{align*}
        & r\left(1 + s \gamma \right) c \phi_2 + c \gamma s \phi_2 \\
        &= r c \phi_2 + \frac{r^2 c^2 \phi_2}{\sqrt{r^2 c^2 \phi_2 + s^2 \phi_2}} + \frac{s^2 \phi_2}{\sqrt{r^2 c^2 \phi_2 + s^2 \phi_2}} \\
        &= r (1 - 4 r^2) + \sqrt{r^2 c^2 \phi_2 + s^2 \phi_2} \\
        &= r (1 - 4 r^2) + r (3 - 4 r^2) = 4 (1 - 2 r^2) r.
    \end{align*}
    Hence, \eqref{eq: third_equation_CpiCpi} is verified. Therefore, the alternate $LG$ path connects the same initial configuration and final location as the initial path. Hence, it remains to be shown that the alternate path is of a lower length than the initial path.

    The length difference between the two paths is given by
    \begin{align*}
        \Delta l (r) &= l_{L_\pi R_\pi} - l_{L_{\phi_1} G_{\phi_2}} = 2 r \pi - r \phi_1 (r) - \phi_2 (r).
    \end{align*}
    It should be noted that, $\Delta l (0) = 0,$ since $\phi_2 = 0$ and \eqref{eq: equation_CpiCpi} is trivially satisfied for any $\phi_1$. Therefore, it suffices to show that $\Delta l(r)$ is an increasing function on $r \in \left[0, \frac{1}{\sqrt{2}} \right]$. 
    
    It should be recalled that the closed-form solution obtained for $\phi_1$ is valid for $r \in \left(0, \frac{1}{\sqrt{2}} \right]$, since $\gamma$ is defined over this range. To ensure that $\phi_1$ is continuous at $r = 0$, consider the expression for $s \gamma$ and $c \gamma.$ For $r \neq 0,$ $s \gamma$ and $c \gamma$ are obtained as given in \eqref{eq: sin_gamma} and \eqref{eq: cos_gamma}, respectively.
    Defining $s \gamma (r = 0) := \frac{1}{3}, c \gamma (r = 0) = \frac{2 \sqrt{2}}{3},$ the solution $\phi_1 = \frac{\pi}{2} + \gamma$ is extended to be continuous on $r \in \left[0, \frac{1}{\sqrt{2}} \right]$. Further, the above solution satisfies \eqref{eq: equation_CpiCpi} for $r = 0$ since the equation is satisfied independent of $\phi_1$ for $r = 0$. Therefore, using this extended solution of $\phi_1$, $\Delta l (r)$ can be expanded as
    \begin{align*}
        \Delta l (r) &= 2 r \pi - r \left(\frac{\pi}{2} + \tan^{-1} \left(\frac{1 - 4 r^2}{2 \sqrt{2} \sqrt{1 - 2 r^2}} \right) \right) \\
        & \quad\, - \tan^{-1} \left(\frac{2 \sqrt{2} r \sqrt{1 - 2 r^2}}{1 - 4 r^2} \right).
    \end{align*}
    The above expression is differentiable for $r \in \left[0, \frac{1}{2} \right) \bigcup \left(\frac{1}{2}, \frac{1}{\sqrt{2}} \right)$. Further, since $\phi_1 \left(\frac{1}{2} \right) = \frac{\pi}{2}, \phi_1 \left(\frac{1}{\sqrt{2}} \right) = 0, \phi_2 \left(\frac{1}{2} \right) = \frac{\pi}{2}, \phi_2 (\frac{1}{\sqrt{2}}) = \pi,$
    \begin{align*}
        \Delta l \left(\frac{1}{2} \right) &= \pi - \frac{\pi}{4} - \frac{\pi}{2} > 0, \\
        \Delta l \left(\frac{1}{\sqrt{2}} \right) &= \sqrt{2} \pi - 0 - \pi > 0.
    \end{align*}
    Hence, it suffices to show $\left(\Delta l (r) \right)' > 0$ for $r \in \left[0, \frac{1}{2} \right) \bigcup \left(\frac{1}{2}, \frac{1}{\sqrt{2}} \right)$. Differentiating the expression for $\Delta l (r)$ and simplifying,
    \begin{align} \label{eq: expression_delta_lprime}
    \begin{split}
        \left(\Delta l (r) \right)' &= \frac{3 \pi}{2} - \tan^{-1} \left(\frac{1 - 4 r^2}{2 \sqrt{2 (1 - 2 r^2)}} \right) \\
        & \quad\, - \frac{6 \sqrt{2 (1 - 2 r^2)}}{3 - 4 r^2}.
    \end{split}
    \end{align}
    It is claimed that $\left(\Delta l (r) \right)' > \frac{3 \pi}{2} - \tan^{-1} \left(\frac{1}{2 \sqrt{2}} \right) - 3 > 0$ for the considered range of $r,$ which will be shown by obtaining the maximum of $\tan^{-1} \left(\frac{h_1}{2 \sqrt{2}} \right)$ and $h_2.$ Here, $h_1 = \frac{1 - 4 r^2}{\sqrt{1 - 2 r^2}},$ $h_2 = \frac{6 \sqrt{2} \sqrt{1 - 2 r^2}}{3 - 4 r^2}$. Noting that $h_1$ and $h_2$ are differentiable for the considered range of $r$,
    \begin{align*}
        h_1' &= \frac{2 r \left(4 r^2 - 3 \right)}{\left(1 - 2 r^2 \right)^{1.5}}, \quad h_2' = \frac{12 \sqrt{2} r \left(1 - 4 r^2 \right)}{\left(3 - 4 r^2 \right)^2 \sqrt{1 - 2 r^2}}.
    \end{align*}
    It can be observed that $h_1' \leq 0$ for the considered range of $r$. Since $\tan^{-1} \left(\frac{h_1}{2 \sqrt{2}} \right)$ is a strictly increasing function, its maximum is attained at $r = 0,$ which evaluates to $\tan^{-1} \left(\frac{1}{2 \sqrt{2}} \right)$. On the other hand, $h_2' = 0 \implies r = \frac{1}{2}$. Since $h_2$ at $r = 0, \frac{1}{\sqrt{2}}, \frac{1}{2}$ evaluates to $2 \sqrt{2}, 0,$ and $3,$ respectively, its supremum is $3$ over the considered range of $r$. Substituting the obtained values in \eqref{eq: expression_delta_lprime}, the claim follows. Since $\Delta l (r = 0) = 0, (\Delta l (r))' > 0$ for $r \in \left[0, \frac{1}{2} \right) \bigcup \left(\frac{1}{2}, \frac{1}{\sqrt{2}} \right),$ and $\Delta l (r = \frac{1}{2}) > 0,$ $\Delta l (r = \frac{1}{\sqrt{2}}) > 0,$ the alternate $LG$ path is of a lower length than the $L_\pi R_\pi$ path. Using a similar proof for the $R_\pi L_\pi$ path, the lemma is proved.
\end{proof}

\subsection{Candidate paths for $e = 1$}

The candidate paths for $e = 1$ for motion planning on a sphere with given initial and final configurations were obtained to be $CGC,$ $CCC$, and degenerate paths of the same for $r \leq \frac{1}{2}$ in \cite{3D_Dubins_sphere}. Hence, considering such paths for the motion planning problem with free terminal orientation for $r \leq \frac{1}{2}$ suffices. It is claimed that paths of type $GC$ are not optimal, and paths of type $CCC$ are non-optimal for $r \leq \frac{1}{2}$. To show the first claim, Lemmas~3.4 and 3.5 from \cite{3D_Dubins_sphere} are utilized, which are restated below.

\begin{lemma} \label{lemma: G_segment}
    Suppose a great circular arc is part of an optimal path. Then, $\|\psi (s) \|^2$ = $A^2 (s) + B^2 (s) + C^2 (s) = 1$ throughout the path (Lemma~3.4 in \cite{3D_Dubins_sphere}).
\end{lemma}
\begin{lemma} \label{lemma: axial_vector}
    Suppose ${\hat \Omega}$ is a skew-symmetric matrix with ${\bf u}$ being the unit axial vector of ${\hat \Omega}$. Suppose $R = e^{{\hat \Omega} \phi}$ be a proper rotation. If ${\bf w} \ne 0$ is such that $R {\bf w} = {\bf w}$, then either $\phi=0$ or ${\bf u} = {\bf w}$ (or equivalently, ${\hat \Omega} {\bf w} = {\bf 0}$) (Lemma~3.5 in \cite{3D_Dubins_sphere}).
\end{lemma}

\begin{lemma}
    For $e = 1,$ a path of type $GC$ is non-optimal.
\end{lemma}
\begin{proof}
    Suppose a non-trivial $GC$ is optimal. Then, at the inflection point with arc length $s_1,$ $A (s_1) = 0$. Due to free terminal orientation, $A (L) = 0$. Since $H = e + C + u_g A \equiv 0$ and $C$ is continuous, $C (s_1) = C (L) = -e = -1$. From Lemma~\ref{lemma: G_segment}, $B (s_1) = B (L) = 0.$ Now, consider the differential equations corresponding to $A, B, C$ given in \eqref{eq: adjoint}, which can be written as
    \begin{align} \label{eq: evolution_psi}
        \psi' (s) = \begin{pmatrix}
            \frac{dA}{ds} (s) \\
            \frac{dB}{ds} (s) \\
            \frac{dC}{ds} (s)
        \end{pmatrix} = \underbrace{\begin{pmatrix}
            0 & 1 & 0 \\
            -1 & 0 & u_g \\
            0 & -u_g & 0
        \end{pmatrix}}_{\Omega} \underbrace{\begin{pmatrix}
            A (s) \\
            B (s) \\
            C (s)
        \end{pmatrix}}_{\psi (s)}.
    \end{align}
    For $s \in (s_1, L),$ $u_g = \pm U_{max}$ (depending on whether $C = L$ or $R$). Therefore, $\Omega$ is a constant skew-symmetric matrix. Defining $\hat{\Omega} = \frac{1}{\sqrt{1 + U_{max}^2}} \Omega$ and $\phi = \sqrt{1 + U_{max}^2} (s - s_1),$ where $\phi$ denotes the angle of the $C$ segment, the solution for $\psi$ can be obtained as
    \begin{align} \label{eq: solution_psi}
        \psi (s) = e^{\hat{\Omega} \phi} \psi (s_1).
    \end{align}
    Since $\psi (s_1) = \psi (L) = (0, 0, -1)^T$, from Lemma~\ref{lemma: G_segment}, $\phi = 0$ or $\psi (s_1)$ is an axial vector for $\hat{\Omega}$. However, the former implies that the $GC$ path is trivial, and the latter implies that $|u_g| = U_{max} = 0$, which cannot occur since $U_{max} \neq 0$. Hence, the $GC$ path is not optimal.
\end{proof}

Having shown that the $GC$ path is not optimal for the free terminal orientation problem (independent of $r$), it is now desired to show the non-optimality of a $CCC$ path for $r \leq \frac{1}{2}$. To this end, Lemma~3.3 from \cite{3D_Dubins_sphere} is utilized, which is restated below.

\begin{lemma} \label{lemma: angle_greater_than_pi_CCC}
     If inflection occurs at $s=s_i$ on the optimal path, then $C(s_i) = -1$ and $A(s_i) = 0$. Furthermore, if $s_1, s_2$ are consecutive inflection points corresponding to a CCC path, then $s_2-s_1 > \pi r$, i.e., the middle segment must have a length greater than $\pi r$. (Lemma~3.3 in \cite{3D_Dubins_sphere}). 
\end{lemma}

\begin{lemma}
    For $e = 1,$ an optimal non-trivial $CCC$ path is of type $C_\alpha C_{\pi + \beta} C_{\pi + \beta},$ where $\alpha, \beta > 0$.   
\end{lemma}
\begin{proof}
    Using Lemma~\ref{lemma: angle_greater_than_pi_CCC}, it follows that the angle of the middle $C$ segment is greater than $\pi.$ Hence, it suffices to show that the last $C$ segment has an angle equal to the middle $C$ segment. The proof for the same follows a similar proof to Lemma~3.7 in \cite{3D_Dubins_sphere}, and an outline of the proof is provided here. Suppose the considered path is of type $C_\alpha C_{\phi_2} C_{\phi_3},$ where $\phi_2 > \pi$. Let the inflection points of this path correspond to arc lengths $s_1$ and $s_2$. Then, $A (s_1) = A (s_2) = 0.$ Further, $A (L) = 0$ due to free terminal orientation. Since $H \equiv 0,$ $C (s_1) = C (s_2) = C (L) = -1$ from \eqref{eq: Hamiltonian} and continuity of $C$. Let $B (s_1) = B_0.$ It is claimed that $B_0 \neq 0$ and $B (s_2) = - B(L) = - B_0.$

    First, using the evolution of $A, B,$ and $C$ given in \eqref{eq: adjoint}, it is claimed that $\|\psi (s)\|^2 = A^2 (s) + B^2 (s) + C^2 (s)$ is constant over the path. The claim follows by differentiating $\|\psi (s)\|^2$ and using the expressions for $\frac{dA}{ds}, \frac{dB}{ds}, \frac{dC}{ds}$ from \eqref{eq: adjoint}.
    Hence, $|B (s_1)| = |B (s_2)| = |B (L)|$, since $A^2 (s_1) + C^2 (s_1) = A^2 (s_2) + C^2 (s_2) = A^2 (L) + C^2 (L).$ Suppose $B (s_1) = B (s_2) = B_0.$ Then, $\psi (s_1) = \psi (s_2) = (0, B_0, -1)^T.$ Further, using \eqref{eq: solution_psi}, $\psi (s_2) = e^{\hat{\Omega}_1 \phi_2} \psi (s_1),$ where
    \begin{align*}
        \hat{\Omega}_1 = \begin{pmatrix}
            0 & k_z & 0 \\
            -k_z & 0 & k_x \\
            0 & -k_x & 0
        \end{pmatrix}.
    \end{align*}
    Here, $k_z = r$ and $k_x = \pm \sqrt{1 - r^2}$. From Lemma~\ref{lemma: axial_vector}, it follows that either $(0, B_0, -1)^T$ is an axial vector for $\hat{\Omega}_1$ or $\phi_2 = 0$. The former is not possible since the axial vector for $\hat{\Omega}_1$ is $(\pm \sqrt{1 - r^2}, 0, r)^T,$ and the latter is not possible since the $CCC$ path is not trivial. Therefore, $B (s_1) = - B (s_2) = B_0 \neq 0.$ Using a similar reasoning, $B (L) = - B (s_2) = B_0$, since $\psi (L) = e^{\hat{\Omega}_2 \phi_3} \psi (s_2),$ where
    \begin{align*}
        \hat{\Omega}_2 = \begin{pmatrix}
            0 & k_z & 0 \\
            -k_z & 0 & -k_x \\
            0 & k_x & 0
        \end{pmatrix}.
    \end{align*}
    Since $\psi (L) = e^{\hat{\Omega}_2 \phi_3} \psi (s_2),$ $\psi (s_2) = e^{\hat{\Omega}_1 \phi_2} \psi (s_1),$ and $\psi (s_1) = \psi (L),$ $(e^{\hat{\Omega}_1 \phi_2} - e^{-\hat{\Omega}_2 \phi_3}) \psi (s_1) = 0.$
    Using the Euler-Rodriguez formula for the exponential of $\hat{\Omega}_1$ and $\hat{\Omega}_2$ in the above equation, simplifying, and using the last two equations (refer to Lemma~3.7 in \cite{3D_Dubins_sphere}),
    \begin{align*}
        \begin{pmatrix}
        B_0 & -k_x \\ k_x & B_0 
        \end{pmatrix} \begin{pmatrix}
        \cos \phi_2 - \cos \phi_3 \\ \sin \phi_2 - \sin \phi_3
        \end{pmatrix} = \begin{pmatrix}
        0\\ 0
        \end{pmatrix}.
    \end{align*}
    Since $B_0^2 + k_x^2 \neq 0,$ it follows that $\cos{\phi_2} = \cos{\phi_3}$ and $\sin{\phi_2} = \sin{\phi_3}.$ Hence, $\phi_2 = \phi_3$.
\end{proof}

Hence, an optimal non-trivial $CCC$ path is of type $C_\alpha C_{\pi + \beta} C_{\pi + \beta}$. However, such a path is non-optimal for $r \leq \frac{1}{2}$, which is proved in Lemma~3.8 in \cite{3D_Dubins_sphere}.

\begin{lemma}
    For any $r\in(0,\frac{1}{2}]$, $\alpha>0$, and $\phi \in (0, \pi)$, the path $L_{\alpha}R_{\pi+\phi}L_{\pi+\phi}$ is not optimal; in particular, there is a path of type $RLR$ with smaller length certifying its non-optimality (Lemma~3.8 in \cite{3D_Dubins_sphere}).
\end{lemma}

Hence, for $e = 1,$ the optimal paths for $r \leq \frac{1}{2}$ are $LG, RG, LR, RL$, and degenerate paths of the same. Further, since for $e = 0,$ the optimal path was identified as $CC$ and $C$ for $r \leq \frac{1}{\sqrt{2}}$, Theorem~\ref{theorem: optimal_paths} is proved.

\section{Results and Conclusions}

Given the list of candidate paths for $r \leq \frac{1}{2},$ each path must be generated (if it exists) for a given initial configuration, final location, and $r$ (or $U_{max}$). For this purpose, the rotation matrix corresponding to each segment derived in Appendix~\ref{app: derivation_rot_mat} was used to derive closed-form expressions for the angles of the $LG, RG, LR,$ and $RL$ paths. Consider the initial configuration of the vehicle to be the identity matrix. A random final location was generated on the sphere for illustration and is shown in Fig.~\ref{subfig: Initial_final_configurations}. Setting $r = 0.4,$ it was observed that two $LG, RG,$ and $LR$ paths could connect the initial configuration and final location. The best among the two paths for the $LG$ and $RG$ paths, and the optimal $LR$ path candidate (such that the angle of the second segment is $\geq \pi$) are shown in Fig.~\ref{subfig: Paths_obtained}. The $LG$ path was observed to be the least-length path among the three paths.



\begin{figure}[htb!]
    \centering
    \subfigure[Chosen initial configuration and final location]{\includegraphics[width = 0.35\textwidth]{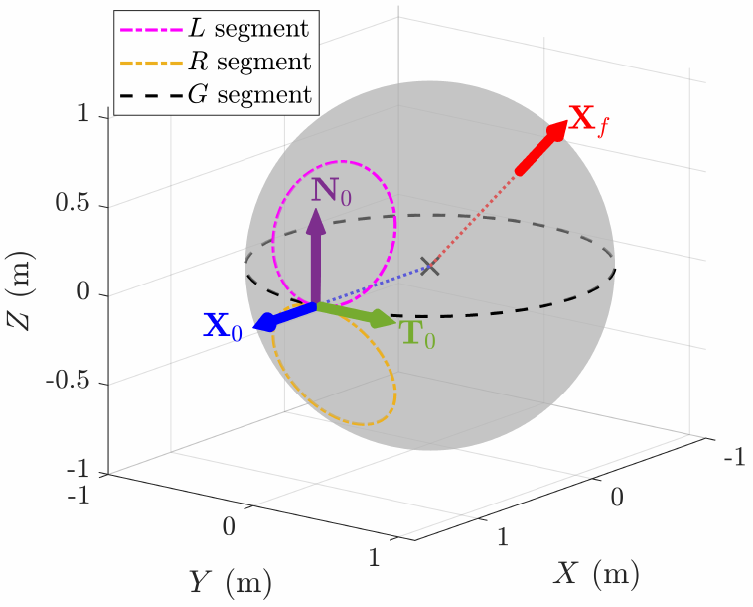}
    \label{subfig: Initial_final_configurations}} \\
    \subfigure[Connecting paths]{\includegraphics[width = 0.35\textwidth]{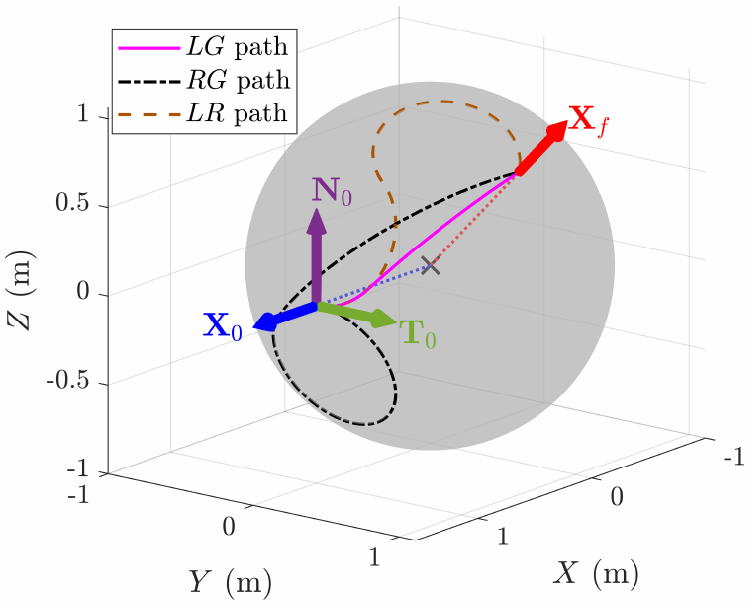}
    \label{subfig: Paths_obtained}}
    \caption{Paths connecting chosen initial configuration of $I$ and final location of $(0.6942, 0.5498, 0.4646)^T$ for $r = 0.4$}
    \label{fig: param_study_teardrop_maneuver}
\end{figure}

In this paper, a motion planning problem on a sphere for a given initial configuration and final location was addressed for a Dubins vehicle. To this end, the vehicle's motion was modeled to be constrained by a maximum geodesic curvature $U_{max},$ which correspondingly constrains the radius of a tight turn on the sphere. Using Pontryagin's Minimum Principle, the optimal paths were shown to be of type $CG, CC, C, G$ for $r \leq 
\frac{1}{2}.$ Future work includes extending the list of candidate paths for $r > \frac{1}{2}.$

\bibliographystyle{IEEEtran}
\bibliography{cite}

\appendix

\subsection{Derivation of rotation matrices for sphere segments} \label{app: derivation_rot_mat}

Consider the differential equations given in \eqref{eq: X_evolution}, \eqref{eq: T_evolution}, and \eqref{eq: N_evolution}. As $u_g$ is piecewise constant, the differential equations for each segment can be represented as
\begin{align}
    \mathbf{R}' (s) &= \mathbf{R} (s) \underbrace{\begin{pmatrix}
        0 & -1 & 0 \\
        1 & 0 & -u_g \\
        0 & u_g & 0
    \end{pmatrix}}_{\Omega},
\end{align}
where $\mathbf{R} = \begin{bmatrix}
    \mathbf{X} (s) & \mathbf{T} (s) & \mathbf{N} (s)
\end{bmatrix}$. The solution to the above differential equations can be obtained as
\begin{align}
    \mathbf{R} (s) = \mathbf{R} (s_i) \underbrace{\left(e^{\Omega^T \Delta s} \right)^T}_{\mathbf{R}_{S}},
\end{align}
where $\Delta s = s - s_i.$ The expression for $e^{\Omega^T \Delta s}$ can be obtained using the Euler-Rodriguez formula for exponential of a skew-symmetric matrix. For this purpose, define $\hat{\Omega} = \frac{1}{\sqrt{1 + u_g^2}} \Omega,$ and $\phi$ such that $\phi = \Delta s \sqrt{1 + u_g^2}.$ Hence, $\hat{\Omega} \phi = \Omega \Delta s$. It should be noted that $\phi$ represents the angle of the segment. Obtaining the expression for $e^{\Omega^T \Delta s},$ $\mathbf{R}_{S}$, where $S \in \{G, L, R\}$, is obtained as 
\begin{align*} 
    \mathbf{R}_G (\phi) &= \begin{pmatrix}
        c \phi & - s \phi & 0 \\
        s \phi & c \phi & 0 \\
        0 & 0 & 1
    \end{pmatrix}, \\
    \mathbf{R}_L (r, \phi) &= \begin{pmatrix}
        \beta_{11} & - r s \phi & \beta_{13} \\
        r s \phi & c \phi & - \beta_{23} \\
        \beta_{13} & \beta_{23} & \beta_{33}
    \end{pmatrix}, \\
    \mathbf{R}_R (r, \phi) &= \begin{pmatrix}
        \beta_{11} & - r s \phi & -\beta_{13} \\
        r s \phi & c \phi & \beta_{23} \\
        -\beta_{13} & -\beta_{23} & \beta_{33}
    \end{pmatrix}.
\end{align*}
In the above equation,
\begin{align*}
    \beta_{11} &= 1 - (1 - c \phi) r^2, \quad \beta_{13} = (1 - c \phi) r \sqrt{1 - r^2}, \\
    \beta_{23} &= s \phi \sqrt{1 - r^2}, \quad \beta_{33} = c \phi + (1 - c \phi) r^2.
\end{align*}

\end{document}